\documentclass[12pt,a4paper]{amsart}
\usepackage{amsmath,amsthm,amssymb}

\usepackage[shortlabels]{enumitem}

\usepackage[totalwidth=15.75cm,totalheight=22.275cm]{geometry}

\usepackage{graphicx}

\usepackage{hyperref}
\usepackage[flushmargin]{footmisc}

\usepackage{gensymb}

\numberwithin{equation}{section}

\makeatletter
\def\swappedhead#1#2#3{%
  \thmnumber{\@upn{\the\thm@headfont #2\@ifnotempty{#1}{.~}}}%
  \thmname{#1}%
  \thmnote{ {\the\thm@notefont(#3)}}}
\makeatother
\swapnumbers

\theoremstyle{plain}
\newtheorem{thm}{Theorem}[section]%
\newtheorem{prop}[thm]{Proposition}%
\theoremstyle{definition}
\newtheorem{question}[thm]{Question}%
\newtheorem{defn}[thm]{Definition}%

\newtheoremstyle{claimstyle}%
   {}
   {}
   {\normalfont}
   {}
   {\itshape}
   {.}
   { }
   {\thmnote{#3}}
   
\theoremstyle{claimstyle}
\newtheorem*{varclaim}{}

\newenvironment{remark}[1][Remark]{\begin{varclaim}[#1]}{\end{varclaim}}
\newenvironment{claim}[1][Claim]{\begin{varclaim}[#1]}{\end{varclaim}}

\newenvironment{subproof}{\begin{proof}}{%
               \end{proof}}

\newcommand*{\defeq}{\mathrel{\vcenter{\baselineskip0.5ex \lineskiplimit0pt
                     \hbox{\scriptsize.}\hbox{\scriptsize.}}}%
                     =}

\renewcommand{\theta}{\vartheta}
\renewcommand{\phi}{\varphi}

\newcommand{\C}{{\mathbb{C}}}
\newcommand{\Ch}{\hat{\C}}

\newcommand{\R}{{\mathbb{R}}}

\newcommand{\classB}{\mathcal{B}}

\newcommand{\deriv}{\mathrm{d}}

\title{Spiders' webs in the Eremenko--Lyubich class}

\begin{document} 

\author{Lasse Rempe} 
\address{Department of Mathematics \\
	 The University of Manchester \\
   Manchester M13 9PL\\
   UK \\ 
	 \textsc{\newline \indent 
	   \href{https://orcid.org/0000-0001-8032-8580%
	     }{\includegraphics[width=1em,height=1em]{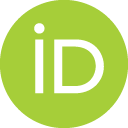} {\normalfont https://orcid.org/0000-0001-8032-8580}}
	       }}
\date{\today}
\email{lasse.rempe@manchester.ac.uk}

\begin{abstract}Consider the entire function $f(z)=\cosh(z)$. We show that
the escaping set $I(f)$~-- that is, the set of points whose orbits tend to infinity under
iteration of $f$~-- has a structure known as a ``spider's web''. This disproves 
  a conjecture of Sixsmith from 2020. 
 In fact, we show that the \emph{fast escaping set} $A(f)$, 
 i.e.\ the subset of $I(f)$ consisting of points whose orbits tend to infinity at an
 iterated exponential rate, is a spider's web. This answers a question of 
 Rippon and Stallard from 2012. We also discuss a wider class of functions to which our results 
 apply, and state some open questions.
  \end{abstract}

\maketitle

\section{Introduction}
 Let $f\colon \C\to\C $ be a transcendental entire function. The set 
    \[ I(f)\defeq \{z\in \C\colon \lim_{n\to\infty} f^n(z) = \infty\} \]
  is called the \emph{escaping set} of $f$, where 
     \[ f^n = \underbrace{f\circ\dots\circ f}_{n\text{ times}}\] 
    denotes the $n$-th iterate of $f$. 
    
   The \emph{fast escaping set} $A(f)\subset I(f)$ consists
     of those points for which $\lvert f^n(z)\rvert $ tends to infinity
    at the fastest possible rate. More precisely, let 
       \[ M(R)\defeq M(R,f)\defeq \max_{\lvert z\rvert \leq R}
          \lvert f(z)\rvert \] denote the maximum modulus function of $f$, and
          write $M^n(R)=M^n(R,f)$ for the $n$-th iterate of this function.
          If $R$ is sufficiently large, then $M(r)>r$ for all $r\geq R$; it follows that
          $M^n(R)\to\infty$. Let us fix such an $R$, and define 
      \[ A_R(f) \defeq \{z\in\C \colon \lvert f^n(z)\rvert \geq M^n(R) \text{ for all $n\geq 0$}\} \]
          and 
        \[ A(f) \defeq \bigcup_{n\geq 0} f^{-n}(A_R(f)). \]
      It can be shown that the set $A(f)$ is independent of the choice of $R$;
        see~\cite[Theorem~2.2]{fastescapingpoints}. 
        
   The escaping set and the fast escaping set have been central objects of study in transcendental dynamics 
    in the past decades; see~\cite{escapingsetsurvey} for a survey on the escaping set, its history and its properties. 
    One of the reasons that the escaping set is useful in the study of transcendental
    dynamics is that it often contains structures that can be used to
    facilitate the study of the overall dynamics. 
    In particular, Rippon and Stallard discovered that, for a large class of entire functions $f$
    (in particular, for many functions of \emph{small growth}), the escaping set, and even 
    $A_R(f)$, has the
    structure of a ``spider's web''. (See~\cite[Theorem~1.9]{fastescapingpoints} for
    a precise statement, and~\cite[\S7.4]{escapingsetsurvey} for further discussion.) 
    
 \begin{defn}[{\cite[Definition~1.2]{fastescapingpoints}}]\label{defn:spidersweb}
   A set $E\subset\C$ is called a \emph{spider's web} if $E$ is connected, and 
    if there exists an increasing
    sequence $(G_n)_{n=0}^{\infty}$ of simply-connected domains with
    $\bigcup_{n=0}^{\infty} G_n =\C$ and $\partial G_n\subset E$ for all $n\geq0$. 
 \end{defn}

 We note that the condition that $A_R(f)$ is a spider's web is independent of the
   choice of $R$; see~\cite[Lemma~7.1]{fastescapingpoints}. Moreover,
   if $A_R(f)$ is a spider's web, then so is $A(f)$, and if $A(f)$ is a spider's
   web, then so is $I(f)$ 
   (see Theorem~1.4  of~\cite{fastescapingpoints} and the remark after its proof). 
              

  The \emph{Eremenko--Lyubich class} consists of those entire functions $f$
    for which the set of critical and asymptotic values is bounded; 
    see~\cite{EremenkoLyubich1992} and~\cite{sixsmith-survey}. 
    If $f\in\classB$, then $f$ is bounded on a curve tending to infinity, and hence
    $A_R(f)$ is never a spider's web (with $R$ as above);
    see~\cite[Theorem~1.8]{fastescapingpoints}. In particular, 
    $A_R(f)$ does not separate the plane (see \cite[Theorem~1.4]{fastescapingpoints}). It can be shown that
    $A_R(f)$ has uncountably many connected components for sufficiently large $R$ 
    (this follows e.g.\ from~\cite[Proposition~8.1]{arclike} and~\cite[Theorem~1.1]{rigidity}; we omit the details). 
    For many $f\in\classB$, even $I(f)$ has uncountably many
    connected components~\cite[Corollary~1.6]{geometricallyfinite}. 
    On the other hand, there are also many values
    of $a$ for which the escaping set of the exponential map 
    $E_a(z) \defeq e^z+a$ is connected~\cite{escapingconnected}.
    However, $I(E_a)$ is never a spider's web. Indeed, 
    $a$ is a logarithmic asymptotic value of $E_a$, and 
    Sixsmith~\cite[Theorem~1.4]{Sixsmith2020} 
    proved that
    $I(f)$ is never a spider's web if $f\in\classB$ has a finite logarithmic asymptotic value.
    
    In view of these results,
     Sixsmith~\cite[Conjecture after Theorem~1.4]{Sixsmith2020} conjectured that $I(f)$ 
    is not a spider's web for any $f\in\classB$. We disprove this conjecture, even for the fast escaping set $A(f)$. 
    
    \begin{thm}[A spider's web escaping set in the Eremenko--Lyubich class]\label{thm:main}
      The set $A(\cosh)$ is a spider's web. 
    \end{thm} 
    The function $\cosh$ has critical values $1$ and $-1$, and no asymptotic values;
    hence it belongs to the class $\classB$. 
       Theorem~\ref{thm:main} 
   also provides the first example of a function for which $A(f)$ is a spider's web,
    but $A_R(f)$ is not. This answers a question of Rippon and Stallard~\cite[p.807, Question~2]{fastescapingpoints}. As we briefly discuss in Section~\ref{sec:generalisations},
    the result holds, with the same proof, for a larger class of transcendental entire functions.
    In particular, the same result is true with $\cosh$ replaced by
    $z\mapsto (\cosh z)^2$, which was studied in~\cite[Section~5]{fastescapingpoints}
    and~\cite[Example~5.10]{pardo-simon-cosine}. 
    
    The idea of the proof of Theorem~\ref{thm:main} is as follows.
     It follows from an observation by Sixsmith (see Theorem~\ref{thm:sixsmith}) that $A(f)$ is a spider's web if 
     and only if $\Ch\setminus A(f)$ is disconnected, where $\Ch=\C\cup\{\infty\}$ denotes
     the Riemann sphere. Pardo-Sim\'on~\cite{pardo-simon-cosine} has given 
     a topological model for the dynamics of the map $h(z)\defeq \cosh(z)$ in terms of the simpler
     ``disjoint-type'' function $g(z) \defeq \cosh(z)/2$, whose dynamics is well-understood. 
      In particular, it can be deduced from her work that there is a 
     continuous and surjective map $\Ch\setminus A(h) \to (J(g)\setminus A(g))\cup\{\infty\}$; see Proposition~\ref{prop:Acoshcomplement}. It was previously shown by Evdoridou and Sixsmith~\cite{nonescapingendpoints} that
     the latter set is disconnected; in fact, the union of $A(g)$ with 
     the Fatou set $F(g)$ is a spider's web. 
     
  (Recall that the \emph{Fatou set} $F(f)\subset \C$ of an entire function $f$ is the set of points where the family of iterates 
    $(f^n)_{n=0}^{\infty}$ is
    equicontinuous with respect to the spherical metric on $\Ch$; intuitively, it is the set of points where the dynamics is
    stable under small perturbations. The complement $J(f) \defeq \C\setminus F(f)$ is called the \emph{Julia set}.
    For $f\in\classB$, we always have $J(f) = \overline{I(f)}=\overline{A(f)}$; see Theorems~5.5, 6.1 and~6.2 of~\cite{escapingsetsurvey}.)

   For the map $g$, the set 
     $F(g)$ is an unbounded connected set, consisting of the immediate basin of an attracting fixed point. In particular,
      the set $A(g)\cup F(g)$ is not a subset of the escaping set, and $A(g)$ and $I(g)$ are not spiders' webs. On the other hand,
      we have $J(h)=\C$ (see Section~\ref{sec:main}), and hence $I(h)$ and $A(h)$ are dense in $\C$. 
     
   \subsection*{Acknowledgements} I thank 
   Walter Bergweiler, Vasiliki Evdoridou, Leticia Pardo-Sim\'on, Phil Rippon, Dave Sixsmith and Gwyneth
    Stallard for interesting discussions about spider's web escaping sets and
    Theorem~\ref{thm:main}.

\section{Iterated exponential growth}
 For a large class of transcendental entire functions, 
   including exponential and trigonometric functions, 
   points in the fast escaping set are characterised as those
   whose orbits exhibit \emph{iterated exponential growth}. Here we briefly review
   this fact and the properties of iterated exponential growth for the reader's
   convenience. 
   
  Let us define $F\colon [0,\infty)\to [0,\infty); t\mapsto \exp(t)-1$.  Then $F(t)>t$ for $t>0$, and hence 
    the sequence $F^n(t)$ tends to infinity. We are interested in the rate at which 
    these orbits
    grow. 
   
 \begin{defn}[Iterated exponential growth]
   A sequence $(a_n)_{n=0}^{\infty}$ of non-negative real numbers has
     \emph{iterated exponential growth} if
        \[ 0 < \liminf_{n\to\infty} F^{-n}(a_n) \leq \limsup_{n\to\infty} F^{-n}(a_n) < \infty.\]
 \end{defn}
 
 The specific function $F$ is not relevant here; 
  any exponentially growing function gives rise to the same
  notion of iterated exponential growth:
 
 \begin{prop}[Elementary properties of iterated exponential growth]\label{prop:exponentialgrowth}
   \mbox{} \begin{enumerate}[(a)]
     \item Let $\delta>0$ and define $\Omega_\delta(t)\defeq \exp(\delta t)$ for $t\in\R$.
       Let $t_0$ be such that $\Omega_{\delta}(t)>t$ for $t\geq t_0$. Then the sequence
       $(\Omega_{\delta}^n(t_0))_{n=0}^{\infty}$ has iterated exponential growth.
       \label{item:differentexponentials}
    \item Let $C>1$. A sequence $(a_n)_{n=0}^{\infty}$ has iterated
        exponential growth if and only if the sequence $\bigl(a_n^C\bigr)_{n=0}^{\infty}$ has iterated
        exponential growth.\label{item:powergrowth}
   \end{enumerate}
 \end{prop}
 \begin{proof}
   Observe that $\Omega_1(t)>F(t)$ for all $t\geq 0$. Let $t_0$ and 
   $\delta$ be as in~\ref{item:differentexponentials}. There 
   is $T\geq t_0$ such that $\Omega_{\delta}^n(T) > \Omega_1^n(1) > F^n(1)$ for $n\geq 0$
     (see e.g.~\cite[Lemma~3.4]{devaneyhairsfastescaping}). Fix $k$ so large that
     $\Omega_\delta^k(t_0) \geq T$; then 
       \[ F^{-n}(\Omega_\delta^n(t_0)) \geq F^{-n}(\Omega_{\delta}^{n-k}(T)) 
            \geq F^{-k}(1)>0 \] 
      for $n\geq k$. Similarly, $\Omega_{1/2}(t)<F(t)$ for $t\geq 1$. Taking
       $T\geq 1$ such that $\Omega_{1/2}^n(T)> \Omega_{\delta}^n(t_0)$ for $n\geq 0$, 
       we see that 
       \[ \limsup F^{-n}(\Omega_{\delta}^n(t_0)) < T < \infty. \]

    To prove part~\ref{item:powergrowth}, observe that $(a_n)_{n=0}^{\infty}$ has
      iterated exponential growth if and only if $(F(a_n))_{n=0}^{\infty}$ does, and 
      the latter sequence is larger than $a_n^C$ when $\lvert a_n\rvert$ is sufficiently
      large. 
  \end{proof} 
    
   It follows easily that, for a trigonometric function such as $h(z) = \cosh z$ or
    $g(z) = \cosh(z)/2$ (the two functions that will play an important role in the 
    proof of Theorem~\ref{thm:main}), a point $z$ is in the fast escaping set if and only if
    its orbit has iterated exponential growth. 
 
    This is true far more generally. 
    A transcendental entire function has \emph{finite order} if there is a constant $C>0$
     such that
       \begin{equation}\label{eqn:order} M(r,f) \leq \exp( r^C) \end{equation}
     for all sufficiently large $r$. It has \emph{positive lower order} if there is a constant
       $c>0$ such that 
       \begin{equation}\label{eqn:lowerorder} M(r,f) \geq \exp(r^c) \end{equation}
       for all sufficiently large $r$. 
  \begin{prop}[Fast escaping points of finite-order functions]\label{prop:classBgrowth}
    If $f$ has finite order and positive lower order, then 
      $z\in A(f)$ if and only if $\lvert f^n(z)\rvert$ has iterated exponential growth.
      In particular, this holds whenever $f\in\classB$ has finite order. 
  \end{prop} 
  \begin{proof}
     If~\eqref{eqn:order} holds and $z\in\C$, then 
       \[ \lvert f^n(z)\rvert \leq M^n(r,f) \leq  (\Omega_C^n(r^C))^{1/C} \] 
   for sufficiently large $r$. Similarly, if~\eqref{eqn:lowerorder} holds, then
      \[ M^n(r,f) \geq (\Omega_c^n(r^c))^{1/c}. \] 
      So the claim follows from Proposition~\ref{prop:exponentialgrowth} and the definition of $A(f)$. 
   
  It is well-known that functions in the Eremenko--Lyubich class have positive lower order; see e.g.\ 
     \cite[Corollary~1.2]{eremenko-lyubich-constant}. This implies the final claim. 
  \end{proof} 

  \section{Proof of the theorem}\label{sec:main}
   Define the entire functions $h$ and $g$ by $h(z) \defeq \cosh(z)$ and $g(z) \defeq \cosh(z)/2$. 
    We begin with some observations. 
    \begin{enumerate}[(a)]
      \item Both functions belong to the Eremenko--Lyubich class. Indeed,
         $h(z) = (e^z+e^{-z})/2$ is the composition of the exponential map $z\mapsto e^z$ and the
         rational map $w\mapsto (w+1/w)/2$ (the \emph{Joukouwsky map}). 
         The former has no critical values and an omitted asymptotic value
         at $0$, while the latter has critical values $1$ and $-1$, and maps $0$ to $\infty$. It follows that 
         $h$ has critical values $1$ and $-1$ and no finite asymptotic values, while
         $g$ has critical values $1/2$ and $-1/2$ and no finite asymptotic values. 
       \item Clearly both functions have finite order in the sense of~\eqref{eqn:order}. 
       \item The map $g$ has two real fixed points $p_a<p_r$,
         with $p_a\approx 0.589$ attracting and $p_r\approx  2.127$ repelling. 
         The interval $(-p_r,p_r)$ belongs to the immediate basin of attraction of
         $p_a$; in particular both critical values belong to this basin. 
         Hence
         $g$ is of ``disjoint type'' in the sense of~\cite[Remark after Definition~2.2]{rigidity};
         see~\cite[Proposition~2.8]{mihaljevic-semiconjugacies}. 
       \item Both critical values of $h$ belong to the real axis. Points on the
         real axis escape to $+\infty$, and $h(x)$ has exponential growth as $x\to\infty$. 
         So $\R\subset A(h)$ by 
         the previous section. It follows that the Julia set of $h$ is the entire complex plane
         (see Theorems~7, 12 and~15 in \cite{waltersurvey}). 
    \end{enumerate}
  The dynamics of the function $g$ is well-understood.
    We refer to~\cite[\S2]{escapingsetsurvey}
    for an elementary discussion of the dynamics of $z\mapsto \sin(z)/2$, which 
     was studied already by Fatou~\cite{fatou}, and which is known to be 
     topologically conjugate to $g$. We require the 
     following fact regarding the complement of $A(g)$. Denote the
     closure of $J(g)$ in $\Ch$ by 
     \[ \hat{J}(g) \defeq J(g)\cup\{\infty\}. \] 
 
  \begin{thm}[The complement of $A(g)$]\label{thm:disjointtypedisconnected}
    The set $\hat{J}(g)\setminus A(g)$ is disconnected. 
  \end{thm} 
  \begin{proof}
    This follows from a theorem of Evdoridou and Sixsmith \cite[Theorem~1.1(b)]{nonescapingendpoints}, 
     which generalises
     a result of Evdoridou and the author for exponential maps~\cite{EvdoridouRempe2018}. They prove that, for any disjoint-type
     entire function $f$ of finite order, the set
      $X \defeq \hat{J}(f) \setminus A(f)$ is \emph{totally separated},
      which means that for any two points $z,w\in X$, there exists an open and closed
      subset $U$ of $X$ with $z\in U$ and $w\notin U$. The Julia set $J(g)$ is non-empty and contains a dense set 
      of periodic points (see~\cite[Theorems~3 and~4]{waltersurvey}). In particular, 
      $X\setminus\{\infty\} = J(g)\setminus A(g)\neq\emptyset$, and  $X$ is disconnected.
  \end{proof} 

 In~\cite{pardo-simon-splitting} and~\cite{pardo-simon-cosine}, Pardo-Sim\'on
  shows that the dynamics of $h$ on $J(h)=\C$ can be understood using the dynamics 
  of $g$ on its Julia set. More precisely, she modifies
  $J(g)$, and the dynamics of $g$ thereon, in an appropriate manner 
    to give a complete model for the topological dynamics of $\cosh$, 
    and indeed for a much more general class of functions.
    We do not describe the construction directly, but 
    will explain how Pardo-Sim\'on's results and their proofs imply
    the following key fact. 

\begin{prop}[The complement of $A(h)$]\label{prop:Acoshcomplement}
 There exists a continuous and surjective function  
   $\psi\colon \Ch\setminus A(h) \to  \hat{J}(g)\setminus A(g)$ 
   with $\psi(z) =\infty$ if and only if $z=\infty$. 
\end{prop}       
\begin{proof}
 In~\cite[Definition~5.5]{pardo-simon-splitting}, 
   Pardo-Sim\'on introduces a space
   $J(g)_{\pm}=J(g)\times \{+,-\}$, with a certain topology that is 
   locally compact and Hausdorff, but not second countable and hence not metrizable.
   We do not require a full description of the topology, but rather will be using
   the following facts (compare Proposition~5.7 and Lemma~5.8 of \cite{pardo-simon-splitting}).
    \begin{enumerate}[(1)]
      \item The projection $\pi\colon J(g)_{\pm}\to J(g)$ is continuous.
      \item   The space $J(g)_{\pm}$ has a one-point compactification
          $\hat{J}(g)_{\pm}= \hat{J}(g)_{\pm}\cup\{\infty\}$, and $x_n\to\infty$ in $\hat{J}(g)_{\pm}$ if and only if $\pi(x_n)\to\infty$ in $\hat{J}(g)$. In particular, $\pi$ extends
        continuously to a map $\pi\colon \hat{J}(g)_{\pm}\to \hat{J}(g)$. 
    \end{enumerate}
   We define $\tilde{g}\colon J(g)_{\pm}\to J(g)_{\pm}; \ (z,\sigma)\mapsto (g(z),\sigma)$. Compare~\cite[p.13414]{pardo-simon-splitting}. 
   
  The function $h$ is \emph{strongly postcritically separated} in the sense 
   of~\cite[Definition~4.1]{pardo-simon-splitting} (see~\cite[p.~1813]{pardo-simon-orbifolds}). 
 By~\cite[Theorem~6.5]{pardo-simon-splitting}, 
  there is a continuous and surjective 
  function $\phi\colon J(g)_{\pm}\to \C$ with 
     $h\circ \phi = \phi\circ \tilde{g}$. 
     Moreover (see~\cite[Formula~(6.21)]{pardo-simon-splitting}),
     $\phi(x_n)\to\infty$ if and only if $x_n\to\infty$ in $J(g)_{\pm}$. In particular,
     $\phi$ extends continuously to a surjection 
     $\phi\colon \hat{J}(g)_{\pm} \to \Ch$ with $\phi(\infty)=\infty$. 

 \begin{claim}[Claim 1]
    Let $x\in J(g)_{\pm}$. Then $\phi(x)\in A(h)$ 
    if and only if $\pi(x)\in A(g)$. 
\end{claim}
\begin{subproof}
   By the discussion following formula~(6.15) on p.13420 of~\cite{pardo-simon-splitting}, the distance
    between the point $\phi(x)$ and
    the point $\pi(x)$ is uniformly bounded by a constant $C$ independent of $x$. 
     Here the distance is measured
     with respect to the hyperbolic metric on a certain hyperbolic orbifold~$\mathcal{O}$; see~\cite[\S4]{pardo-simon-splitting}. 
     We refer to~\cite[\S3]{pardo-simon-orbifolds} for background on hyperbolic orbifolds. 
  
 The underlying surface of the orbifold $\mathcal{O}$ is a subset of $\C$, and hence has at least a 
   puncture at $\infty$. (See~\cite[Theorem~4.6]{pardo-simon-splitting}.) Therefore the density $\rho_{\mathcal{O}}$ of the hyperbolic
   metric $\rho_{\mathcal{O}}(z)\lvert \deriv z\rvert$ of $\mathcal{O}$ satisfies 
     \begin{equation}\label{eqn:orbifoldmetric} \rho_{\mathcal{O}}(z) > \frac{1}{2\lvert z\rvert \log\lvert z\rvert} \end{equation}  for sufficiently large $z$. Indeed, consider the sub-orbifold $\tilde{\mathcal{O}}$ of $\mathcal{O}$ consisting of points
     of modulus greater than $1$. Then $\tilde{\mathcal{O}}$ is contained in $W\defeq \{\lvert z\rvert > 1\}$ (i.e., the inclusion map $\tilde{O}\to W$ is a holomorphic map between orbifolds; see \cite[Definition~3.2]{pardo-simon-orbifolds}). Hence 
     \[ \rho_{\tilde{\mathcal{O}}}(z) \geq \rho_W(z) = \frac{1}{\lvert z\rvert \log \lvert z\rvert} \]
     when $\lvert z\rvert > 1$ by Pick's theorem (see~\cite[Corollary~3.6]{pardo-simon-orbifolds}. Moreover, as $z\to\infty$, we have $\rho_{\tilde{\mathcal{O}}}(z)/\rho_{\mathcal{O}}(z)\to 1$;
     see~\cite[Theorem~1.5]{pardo-simon-orbifolds}.  This establishes~\eqref{eqn:orbifoldmetric}.
     
   It follows that 
   there exists a constant $\delta>1$ such that the hyperbolic distance in $\mathcal{O}$ between $\{\lvert z\rvert = R\}$ and $\{\lvert z\rvert = R^{\delta}\}$ 
   is greater than $C$, for sufficiently large $R$. (A simple calculation shows that we can take $\delta = \exp(2C)$.)  
   
    Applying this fact to the orbits $h^n(\phi(x)) = \phi(\tilde{g}^n(x))$ and $g^n(\pi(x)) = \pi(\tilde{g}^n(x))$, we see that 
      \[   \lvert g^n(\pi(x))\rvert^{\frac{1}{\delta}} < \lvert h^n(\phi(x))\rvert < \lvert g^n(\pi(x))\rvert^\delta \]
    for sufficiently large $n$. By Proposition~\ref{prop:exponentialgrowth}, 
   these sequences either both have iterated exponential growth, or neither does.
   The claim follows by Proposition~\ref{prop:classBgrowth}. \end{subproof} 
     
  In~\cite[\S5]{pardo-simon-cosine}, a complete description is given of when two points of $J(g)_{\pm}$ have the same image under
   $\phi$. We require the 
   following. 
  
  \begin{claim}[Claim~2]
   Suppose that $x,y\in J(g)_{\pm}$ satisfy $\phi(x)=\phi(y)$ and 
    $\pi(x)\notin A(g)$. Then $\pi(x)=\pi(y)$. 
  \end{claim}  
  \begin{subproof}
     Every connected component $s$ of $J(g)$ is an arc connecting a finite
      point $z_{s}\in s$ 
      to infinity~\cite[Theorem~1.3]{pardo-simon-cosine}; compare also~\cite[\S5]{aarts-oversteegen}
      and~\cite[\S2.1]{escapingsetsurvey}. The open arc 
     $\Gamma_s\defeq s \setminus\{z_{s}\}$ is called a
      \emph{dynamic ray} of $g$, and $z_{s}$ is called its \emph{endpoint}. 
      
     Then $\tilde{s} \defeq s \times \{\sigma\}$ is a connected component of $J(g)_{\pm}$ 
      for $\sigma\in \{+,-\}$. We call $\Gamma_{\tilde{s}} \defeq \gamma \times \{\sigma\}$ a 
       \emph{dynamic ray of $\tilde{g}$} and $z_{\tilde{s}}\defeq (z_s,\sigma)$ its endpoint. 

       The map $\phi$ is injective on $\tilde{s}$~\cite[Theorem~6.5]{pardo-simon-splitting}, 
       and the image $\phi(\Gamma_{\tilde{s}})$ is called 
       a ``canonical dynamic ray'' of $h$. 
      We say that the dynamic ray $\phi(\Gamma_{\tilde{s}})$ \emph{lands} at the point
        $\phi(z_{\tilde{s}})$. 
        Every point of $\C = J(h)$ is either on a canonical dynamic ray, or 
        a landing point of such a ray~\cite[Theorem~1.2]{pardo-simon-splitting}.

       Two canonical dynamic rays $\phi(\Gamma_{\tilde{s}_1})$ and $\phi(\Gamma_{\tilde{s}_2})$  
       of $h$ may coincide, but only if $\pi(\tilde{s}_1)=\pi(\tilde{s}_2)$ (i.e., 
       $\tilde{s}_1=s\times \{\sigma_1\}$ and $\tilde{s}_2 = s\times \{\sigma_2\}$ for some
       $\sigma_1,\sigma_2\in\{+,-\}$). 
        By~\cite[Proposition~5.8]{pardo-simon-cosine}, in the particular case of 
        $h=\cosh$, 
        no two different 
        canonical dynamic rays land at the same point. So we conclude that,
        if $\phi(x)=\phi(y)$ and $\pi(x)$ and $\pi(y)$ are endpoints of $J(g)$, then $\pi(x)=\pi(y)$. 
      
     It is known that all points on dynamic rays of $g$ escape at an iterated
     exponential rate, and hence belong to $A(g)$; 
     see~\cite[Proposition~4.3]{rottenfusser-schleicher}. If $\pi(x)\notin A(g)$ and $\phi(x)=\phi(y)$, then also $\pi(y)\notin A(g)$ by Claim~1. Hence both $\pi(x)$ and $\pi(y)$ are endpoints of $J(g)$, and the claim follows. 
     
      (We remark that a dynamic ray cannot land at a point that is itself on a ray, so if $\phi(x)=\phi(y)$, then $\pi(x)$ is an endpoint if and only if $\pi(y)$ is. See~\cite[Corollary~6.7]{rottenfusser-schleicher}. However, we do not require this fact.)
\end{subproof}

  Now we can complete the proof of Proposition~\ref{prop:Acoshcomplement}. It follows from Claims~1 and~2 that setting
    \[ \psi(\phi(x)) \defeq \pi(x) \]
   yields a well-defined surjective function
     \[ \psi \colon \Ch\setminus A(h) \to \hat{J}(g)\setminus A(g). \]

      We claim that $\psi$ is continuous. Continuity at $\infty$ is immediate from
      the continuity properties of $\phi$ and $\pi$ at infinity, noted above. Indeed,
      if $\phi(x_n)\to\infty$, then $x_n\to\infty$ and thus $\pi(x_n)\to\infty$. 

   Now suppose that $z_n = \phi(x_n)\to z$ in $J(h)\setminus A(h)$, and write
     $w_n = \psi(z_n) = \pi(x_n)$ and $w = \psi(z) = \pi(x)$. 
      Since $\hat{J}(g)$ is compact, we may suppose that
      $w_n\to w' \in \hat{J}(g)$, and must prove that $w=w'$. 
      
      By compactness, the sequence $(x_n)$ has an accumulation 
      point $y\in \hat{J}(g)_{\pm}$. By continuity of $\phi$, we see that
      $\phi(y)=z = \phi(x)$. We have $\pi(x)= \psi(z) \in J(g)\setminus A(g)$.
      By Claim~2 and continuity
      of $\pi$, we conclude that 
      $w' = \pi(y)=\pi(x)=w$, as claimed.  \end{proof}

 To complete the proof, we use the following observation, which is due to Rippon and Stallard~\cite[Theorem~2]{small-growth} for $I(f)$ and  Sixsmith~\cite[Theorem~1.5]{Sixsmith2020} for $A(f)$.
 \begin{thm}[Characterisation of an escaping spider's web]\label{thm:sixsmith}
  Let $f$ be a transcendental entire function. Then $A(f)$ (respectively $I(f)$) 
   is a spider's web if and 
   only if it separates some point of $J(f)$ from infinity. 
 \end{thm} 

\begin{proof}[Proof of Theorem~\ref{thm:main}] 
 By Theorem~\ref{thm:disjointtypedisconnected}, the set $\hat{J}(g)\setminus A(g)$ is disconnected. 
   By Proposition~\ref{prop:Acoshcomplement}, this set is the continuous image of
   $\Ch\setminus A(\cosh)$, hence the latter set is also disconnected. 
   
  In particular, $A(\cosh)$ separates some point of
    $\C\setminus A(\cosh)\subset J(\cosh)$ from infinity. The claim now follows
    from Theorem~\ref{thm:sixsmith}.
\end{proof}

 \section{Further remarks and questions}\label{sec:generalisations}
  It is easy to see that $A(\cosh)$ has a dense path-connected
   subset, consisting of the real axis and its iterated preimages. In view of
   our result and the results of Pardo-Sim\'on, 
   it is plausible that, for $A(\cosh)$, the boundaries in Definition~\ref{defn:spidersweb} can be chosen
   to be Jordan curves. This would imply that 
   $A(\cosh)$ and $I(\cosh)$ are in fact path-connected; as far as we are aware,
   this would be the first example of a path-connected escaping set in the class $\classB$. 
   
 \begin{question}
   Are $I(\cosh)$ and $A(\cosh)$ path-connected?
\end{question}

 Our proof of Theorem~\ref{thm:main} 
 generalises immediately to a large class of transcendental entire functions
  with escaping critical values. Indeed, suppose that $f\in\classB$ is a function satisfying
  the following conditions.
   \begin{enumerate}[(1)]
    \item The function $f$ has finite order of growth; i.e.,~\eqref{eqn:order} holds for some $C$. 
       \label{item:finiteorder}
     \item $f$ has finitely many critical values and 
        no finite asymptotic values, and the degree of the critical points of $f$
       is uniformly bounded.\label{item:boundedcriticality}
     \item All critical points of $f$ belong to $A(f)$, and the 
         map $f$ is strongly post-critically separated in the sense of~\cite{pardo-simon-splitting}.\label{item:criticalfastescaping}
     \item No two dynamic rays of $f$ land at a common point.\label{item:rayslanding}
   \end{enumerate}
  \begin{remark}
    We refer to~\cite{pardo-simon-splitting} 
       for the definition of strongly post-critically separated maps, as well as that of
       dynamic rays and their landing points. 
  \end{remark}
  
  \begin{thm}[More fast escaping spiders' webs]\label{thm:general}
    Suppose that $f\in\classB$ satisfies conditions~\ref{item:finiteorder}--\ref{item:rayslanding}. Then $A(f)$ is a spider's web. 
  \end{thm}
  \begin{proof} 
    The proof proceeds exactly as the proof of Theorem~\ref{thm:main}. 
    Indeed, Conditions~\ref{item:finiteorder}, 
    \ref{item:boundedcriticality} and~\ref{item:criticalfastescaping} imply that
    $f$ satisfies the assumptions of~\cite[Theorem~6.5]{pardo-simon-splitting}. 
    This yields a function $\phi\colon \hat{J}(g)_{\pm} \to J(f)$ as discussed in the proof
    of Proposition~\ref{prop:Acoshcomplement}, where      $g(z) = \lambda f(z)$, with
    $\lvert \lambda\rvert$ sufficiently small. Claim~1 in the proof of Proposition~\ref{prop:Acoshcomplement} holds with
    the same proof. 
    Properties~\ref{item:criticalfastescaping} and~\ref{item:rayslanding}
    ensure that Claim~2 also holds. Hence we conclude as before that there is 
    a continuous function $\phi\colon \Ch\setminus A(f)\to \hat{J}(g)\setminus A(g)$.
      
   By~\cite[Theorem~1.1(b)]{nonescapingendpoints}, 
     Condition~\ref{item:finiteorder} again implies that 
     $\hat{J}(g)\setminus A(g)$ is disconnected. So $\Ch\setminus A(f)$ is also disconnected, and
     $A(f)$ is a spider's web by Theorem~\ref{thm:sixsmith}.
  \end{proof}
     
  We may ask whether any of the conditions above can be omitted or 
   weakened. If $f$ has infinite order,
   the set $I(f)\setminus A(f)$ may contain curves to infinity; see Remark~(3) after 
   Theorem~5.2 of~\cite{devaneyhairsfastescaping}. Such an example may be chosen
   also to 
   satisfy condition~\ref{item:boundedcriticality}--\ref{item:rayslanding}, 
   using the techniques of Bishop~\cite{qcfolding}. 

   The function $f(z)=e^z$ satisfies all conditions except 
   for~\ref{item:boundedcriticality}, but it follows from~\cite{devaney}
   that $\C\setminus I(f)$ contains unbounded connected sets. (In fact, as 
   already mentioned in the
    introduction, if $f\in\classB$ has a finite asymptotic value, then $I(f)$ is not a spider's
    web by~\cite[Theorem~1.4]{Sixsmith2020}.)
    
  Finally, the function $g(z)=\cosh(z)/2$ 
    satisfies conditions~\ref{item:finiteorder}, \ref{item:boundedcriticality} 
    and~\ref{item:rayslanding}, but the Fatou set $F(f)$ is connected,
   and $A(f)$ and $I(f)$ have uncountably many connected components. 
   
  So conditions~\ref{item:finiteorder},~\ref{item:boundedcriticality} 
    and~\ref{item:criticalfastescaping} cannot be omitted entirely (although it is likely that some or all of them could be weakened). We do not
    know whether condition~\ref{item:rayslanding} can be omitted in Theorem~\ref{thm:general}. 
    
   \begin{question}
     Suppose that $f\in\classB$ satisfies~\ref{item:finiteorder},~\ref{item:boundedcriticality}
      and~\ref{item:criticalfastescaping}. Is $A(f)$ a spider's web? 
   \end{question}
   
  More generally, we may ask the following.
  
   \begin{question}
     Suppose that $f\in\classB$ has finite order of growth, has no finite asymptotic value 
      and finitely many critical values, and all critical values 
      belong to $A(f)$ (resp.\ $I(f)$). When is $A(f)$ (resp.\ $I(f)$)
      a spider's web? 
   \end{question}

\providecommand{\bysame}{\leavevmode\hbox to3em{\hrulefill}\thinspace}
\providecommand{\MR}{\relax\ifhmode\unskip\space\fi MR }
\providecommand{\MRhref}[2]{%
  \href{http://www.ams.org/mathscinet-getitem?mr=#1}{#2}
}
\providecommand{\href}[2]{#2}

    \end{document}